\documentclass{amsart}
\usepackage{amsmath,epsf,amsfonts}
\usepackage{amssymb}

\newtheorem{theorem}{Theorem}[section]
\newtheorem{acknowledgement}[theorem]{Acknowledgement}

\newtheorem{corollary}[theorem]{Corollary}

\newtheorem{example}[theorem]{Example}

\newtheorem{lemma}[theorem]{Lemma}

\newtheorem{proposition}[theorem]{Proposition}
\newtheorem{remark}[theorem]{Remark}

\input{tcilatex}

\begin{document}
\title{On perspective Abelian groups }
\author{Grigore C\u{a}lug\u{a}reanu, Andrey Chekhlov}
\address{Department of Mathematics, Babe\c{s}-Bolyai University,
Cluj-Napoca, 400084, Romania}
\email{calu@math.ubbcluj.ro}
\address{Department of Mathematics, Faculty of Mechanics and Mathematics,
Tomsk State University, Tomsk, Russia}
\email{cheklov@math.tsu.ru}

\begin{abstract}
As a special case of perspective $R$-modules, an Abelian goup is called 
\textsl{perspective} if isomorphic summands have a common complement. In
this paper we describe many classes of such groups.
\end{abstract}

\subjclass[2020]{Primary: 20K10, 20K15, 20K27. Secondary: 20K21, 16D99.}
\keywords{Abelian group, perspective module, primary group, torsion-free
group}
\maketitle

\section{Introduction}

This paper concerns about direct summands of Abelian groups. To simplify the
writing, $G$ will denote an (arbitrary) Abelian group and by some different
letters we denote direct summands of $G$. In what follows, since \emph{all
summands we consider are direct}, we remove this adjective. Moreover, the
word "complement" will be used only for direct complements. In trying to
make a notation difference, by $\mathbb{Z}(m)$ we denote the Abelian group
and by $\mathbb{Z}_{m}$ we denote the ring of integers modulo $m$. For an
Abelian group $G$, by $\mathrm{End}(G)$ we simply denote $\mathrm{End}_{%
\mathbb{Z}}(G)$, that is, the endomorphism ring of $G$.

We start with the following general

\textbf{Definition} (see \cite{goo}). Let $L$ be a bounded lattice. Two
elements $x,y\in L$ are said to be \textsl{perspective} (in $L$) provided
they have a common (direct) complement, i.e., an element $z\in L$ such that $%
x\vee z=y\vee z=1$, $x\wedge z=y\wedge z=0$. This definition comes back to
John von Neumann.

Specializing for the submodule lattice of a module, two summands $A,B$ of a
module $M$ will be denoted by $A\sim B$, if they have a common complement,
i.e., there exists a submodule $C$ such that $M=A\oplus C=B\oplus C$. It is
clear that $A\sim B$ implies $A\cong B$. A module $M$ is called \textsl{%
perspective} when $A\cong B$ implies $A\sim B$ for any two summands $A,B$ of 
$M$.

A module $_{R}M$ over a ring $R$ is said to satisfy \textsl{internal
cancellation} (or we say $M$ is internally cancellable; \textsl{IC}, for
short) if, whenever $M=K\oplus N=K^{\prime }\oplus N^{\prime }$ (in the
category of $R$-modules), $N\cong N^{\prime }\Rightarrow K\cong K^{\prime }$
[or $M/N\cong M/N^{\prime }$].

It is clear that \emph{perspective modules satisfy the internal cancellation
property }in the sense that complements of isomorphic summands are
isomorphic (see \cite{khu}).

The modules definition can be restricted to rings as follows

\textbf{Definition} A ring $R$ is called \textsl{perspective} if isomorphic
direct summands of $_{R}R$ have a common (direct) complement.

This property for rings turns out to be left--right symmetric, that is, $%
_{R}R$ is perspective if and only if $R_{R}$ is perspective for any ring $R$
and we call such ring a perspective ring.

In this paper, we characterize some large classes of perspective Abelian
groups. In the sequel, the word "group" will always mean an "Abelian group".

Our main results are the following.

1) For large classes of Abelian groups, we show that the perspectivity is
equivalent to finite rank. These are: the divisible primary groups, the
divisible torsion-free groups, the homogeneous completely decomposable
groups with type $\mathbf{t}(\mathbb{Q}_{p})$.

2) The only perspective rank 1 torsion-free groups are the rational groups.
In particular, the only free perspective group is $\mathbb{Z}$.

3) Any finite $p$-group is perspective. Any bounded perspective $p$-group is
finite.

4) A torsion group is perspective iff so are all its primary components.

5) (a) Let $G=T(G)\oplus F$ where $T(G)$ is the torsion part of $G$ and $F$
is torsion-free. Then $G$ is perspective iff $T(G)$ and $F$ are perspective

(b) Let $G=D(G)\oplus R$ where $D(G)$ is the divisible part of $G$ and $R$
is reduced. Then $G$ is perspective iff $D(G)$ and $R$ are perspective.

6) We describe the perspective torsion-free algebraically compact groups.

7) For any rational group $H$ (i.e., rank 1 torsion-free group) we
characterize when $H\oplus H$ is perspective.

Section 2 recalls a few general results on perspective modules from \cite%
{gar}, which are used in the sequel. Section 3 contains our results on
perspective Abelian groups, divided into subsection 3.1, about definitions,
subsection 3.2, about reduction theorems, subsection 3.3, about perspective
torsion groups and subsection 3.4, about perspective torsion-free groups.
Some examples are given in the end.

\section{Generalities on perspective modules}

\textbf{Definition 8.1} \cite{lam1}. Let $\mathcal{P}$ be a module-theoretic
property. We say that $\mathcal{P}$ is an \textsl{endomorphism ring property}
(or \textbf{ER}-property for short) if, whenever $A_{R}$ and $A_{S}^{\prime
} $ are right modules (over possibly different rings $R$ and $S$) such that $%
\mathrm{End}(A_{R})\cong \mathrm{End}(A_{S}^{\prime })$ (as rings), $A_{R}$
satisfies $\mathcal{P}$ implies that $A_{S}^{\prime }$ does (and then, of
course, also conversely).

As an exhaustive reference on perspective $R$-modules we mention \cite{gar}.
However, \emph{not much} will be used when describing the perspective
Abelian groups.

For further reference we list here from \cite{gar}.

\begin{theorem}
\label{3.4}For a module $M_{S}$ with $R=\mathrm{End}_{S}(M)$, the following
conditions are equivalent:

(1) $M$ is perspective.

(4) If $erse=e$ for some $e^{2}=e,r,s\in R$, then $erte=e$ for some $t\in R$
such that $ete\in U(eRe)$.

In particular, $M_{S}$ is perspective iff $R_{R}$ is perspective, i.e.,
perspectivity is an \textbf{ER}-property.
\end{theorem}

We mention here a useful consequence (not recorded in \cite{gar}) of the
previous theorem.

\begin{corollary}
\label{prod}Arbitrary products of perspective rings are perspective.
\end{corollary}

\begin{proof}
Suppose $R$ is direct product of $R_{i}$, $i\in I$ and $erse=e$ for some $%
e^{2}=e,r,s\in R$. As $e=(e_{i})$, $r=(r_{i})$, $s=(s_{i})$ so $erse=e$ iff $%
e_{i}r_{i}s_{i}e_{i}=e_{i}$ for each $i$. As each $R_{i}$ is perspective, so
there exists $t_{i}$ such that $e_{i}r_{i}t_{i}e_{i}=e_{i}$ and $%
e_{i}t_{i}e_{i}\in U(e_{i}R_{i}e_{i})$. If $t=(t_{i})$, then $erte=e$ and $%
ete\in U(eRe)$.
\end{proof}

\begin{proposition}
\label{5.4}Any direct summand of a perspective module is perspective.
\end{proposition}

\begin{corollary}
\label{5.10}If $M$ and $N$ are perspective $R$-modules with $\mathrm{Hom}%
_{R}(M,N)=0$, then $M\oplus N$ is perspective.
\end{corollary}

\begin{remark}
\normalfont\label{5.13}The group $\mathbb{Z}\oplus \mathbb{Z}$ is not
perspective. For example, $(2,5)\mathbb{Z}$ and $(1,0)\mathbb{Z}$ are
isomorphic direct summands of $\mathbb{Z}^{2}$ as a $\mathbb{Z}$-module,
which do not have a common complement.
\end{remark}

From \cite{sch} we recall a \emph{method of constructing common complements}
for some special direct sums.

Let $G=H\oplus K$. A subgroup $D$ of $G$ is called a \textsl{diagonal} in $G$
(with respect to $H$ and $K$) if $D+H=G=D+K$ and $D\cap H=0=D\cap K$.

\begin{theorem}
\label{dia}Let $G=H\oplus K$. If $\delta :H\rightarrow K$\ is an isomorphism
then $D(\delta )=D(H,\delta )=\{x+\delta (x)|x\in H\}=(1+\delta )(H)$\ is a
diagonal in $G$\ (with respect to $H$\ and $K$). Conversely, if $D$\ is a
diagonal in $G$\ (with respect to $H$\ and $K$) there is a unique
isomorphism $\delta :H\rightarrow K$\ such that $D=D(\delta )$.
\end{theorem}

Thus, there is a bijection between the diagonals (with respect to $H$ and $K$%
) and isomorphisms of $H$ and $K$.

Every subgroup $U$ of a direct sum $G=H\oplus K$ belongs to the direct
product $\mathbf{L}=L(H)\times L(K)$ (i.e., has the form $H^{\prime }\oplus
K^{\prime }$ for $H^{\prime }\leq H$ and $K^{\prime }\leq K$) or is a
diagonal.

\section{Perspective Abelian groups}

First about the

\subsection{Definition}

As the Abelian groups analogue for $\mathbb{Z}$-modules, an (Abelian group) $%
G$ is called \textsl{perspective} if isomorphic summands of $G$ have a
common complement.

In symbols: if $G=A\oplus H=B\oplus K$ with $A\cong B$, there exists (a
summand) $C$ such that $G=A\oplus C=B\oplus C$.

\textbf{Nonexample}. Let $N$ be a group such that $N\ncong N\oplus N$ and
let $G=N_{1}\oplus N_{2}\oplus N_{3}\oplus ...$ countably many copies with $%
N_{n}=N$. Then $G$ is \emph{not} IC (and so nor perspective).

Indeed, $H=N_{2}\oplus N_{3}\oplus ...$ and $S=N_{3}\oplus ...$ are
isomorphic summands, but $G/H\cong N\ncong N\oplus N\cong G/S$.

\textbf{Remarks}. 1) Obviously, \emph{if two summands of a group have a
common complement, these are isomorphic}.

Indeed, if $G=H\oplus K=L\oplus K$ then $H\cong G/K\cong L$.

Therefore, perspectivity is a converse of this property.

2) \emph{Indecomposable groups are trivially perspective} (e.g., any
infinite cyclic group, any cocyclic group or rank 1 torsion-free group).

3) We mention (see Proposition \textbf{3.10}, \cite{KMT}) that \emph{for two
idempotent endomorphisms }$\varepsilon $\emph{, }$\delta $\emph{\ of }$G$%
\emph{, }$\varepsilon (G)\cong \delta (G)$\emph{\ iff there exists
endomorphisms }$\alpha $\emph{, }$\beta $\emph{\ of }$G$\emph{\ such that }$%
\varepsilon =\alpha \beta $\emph{\ and }$\delta =\beta \alpha $\emph{. }Also
equivalently, \emph{the left }$\mathrm{End}(G)$\emph{-modules} $\mathrm{End}%
(G)\varepsilon $ \emph{and} $\mathrm{End}(G)\delta $ \emph{are isomorphic to
each other. }Since we deal only with direct summands, equivalently, we can
deal only with the idempotent endomorphisms of the group. Thus, two "\textsl{%
isomorphic}" endomorphisms $\varepsilon $, $\delta $ [i.e., $\mathrm{im}%
(\varepsilon )\cong \mathrm{im}(\delta )$] of a group $G$ are \textsl{%
perspective}, if there is an endomorphism $\gamma $ such that $\mathrm{im}%
(\varepsilon )\oplus \mathrm{im}(\gamma )=G=\mathrm{im}(\delta )\oplus 
\mathrm{im}(\gamma )$.

More general, the group is \textsl{IC} if for any two endomorphisms $%
\varepsilon $, $\delta $ of $G$, $\mathrm{im}(\varepsilon )\cong \mathrm{im}%
(\delta )$ implies $G/\mathrm{im}(\varepsilon )\cong G/\mathrm{im}(\delta )$.

4) The group $G$ is indecomposable iff $\mathrm{End}(G)$ has only the
trivial idempotents (is \textsl{connected}). Such groups are trivially
perspective.

\bigskip

\emph{Apparently} there is another definition one could give for perspective
Abelian groups.

\textbf{Definition}. An Abelian group $G$ is called \textsl{e-perspective}
if its endomorphism ring $\mathrm{End}(G)$ is (left or right) perspective.

However, it follows from Theorem \ref{3.4} that for a module $M_{S}$ with $R=%
\mathrm{End}_{S}(M)$, $M_{S}$ is perspective iff $R_{R}$ is perspective,
i.e., perspectivity is an \textbf{ER}-property. Therefore, for $S=\mathbb{Z}$
it follows that

\begin{proposition}
An Abelian group is perspective iff it is e-perspective.
\end{proposition}

Hence, in the sequel we can use any of these two (equivalent) definitions.

\subsection{Reduction theorems}

\begin{proposition}
\label{dir}Summands of perspective groups are perspective.
\end{proposition}

\begin{proof}
Suppose $G=H\oplus K$ and $H=S\oplus T=L\oplus N$ with $S\cong L$. Since
these are direct summands also in $G$, by hypothesis, there is $M\leq
^{\oplus }G$ such that $G=S\oplus M=L\oplus M$. Then, by modularity of the
subgroup lattice: $H=G\cap H=(\underline{S}\oplus M)\cap \underline{H}%
\overset{\mathrm{\mathrm{mod}}}{=}\underline{S}\oplus (M\cap \underline{H})$
(since $S\leq H$) and similarly $H=L\oplus (M\cap H)$, so $M\cap H$ is a
common complement for $S$ and $L$.
\end{proof}

\begin{proposition}
\label{fu-inv}Let $G=\bigoplus\limits_{i\in I}H_{i}$ where each summand $%
H_{i}$ is fully invariant in $G$. Then $G$ is perspective iff all $H_{i}$, $%
i\in I$, are perspective.
\end{proposition}

\begin{proof}
By Corollary \ref{prod}, arbitrary products of perspective rings are
perspective. It just remains to note that $\mathrm{End}(G)\cong
\prod\limits_{i\in I}\mathrm{End}(H_{i})$ (see Theorem \textbf{106.1}, \cite%
{fuc2}, the $I\times I$ matrices are diagonal).
\end{proof}

\begin{corollary}
\label{tor}Let $G$ be a torsion group. Then $G$ is perspective iff so are
all its primary components.
\end{corollary}

\begin{proof}
A straightforward application of the previous proposition.
\end{proof}

As customarily, this reduces the study of perspective torsion groups to
perspective $p$-groups, for any prime $p$.

\bigskip

We can use Corollary \ref{5.10} whenever $\mathrm{Hom}(G,H)=0$, that is,
for: (i) $G$ torsion, $H$ torsion-free, (ii) $G$ a $p$-group and $H$ a $q$%
-group with different primes $p\neq q$, (iii) $G$ divisible and $H$ reduced.
Thus

\begin{corollary}
\label{red}(a) Let $G=T(G)\oplus F$ where $T(G)$ is the torsion part of $G$
and $F$ is torsion-free. Then $G$ is perspective iff $T(G)$ and $F$ are
perspective.

(b) Let $G=D(G)\oplus R$ where $D(G)$ is the divisible part of $G$ and $R$
is reduced. Then $G$ is perspective iff $D(G)$ and $R$ are perspective.
\end{corollary}

\textbf{Examples}. $\mathbb{Z}_{m}\oplus \mathbb{Z}$, $\mathbb{Z}_{p^{\infty
}}\oplus \mathbb{Z}$ or $\mathbb{Z}_{p^{\infty }}\oplus \mathbb{Q}$, all are
perspective (splitting) mixed groups.

\bigskip

Therefore, the study of splitting mixed perspective (Abelian) groups reduces
to perspective primary groups and to torsion-free groups. Moreover, it
reduces to perspective divisible groups and to reduced groups.

According to these reductions, the study of perspective (Abelian) groups
reduces to reduced perspective $p$-groups, reduced perspective torsion-free
groups and mixed perspective not splitting groups. Some of such results are
proved in the sequel. The reader can convince himself that even for highly
predictable (for Abelian groups theorists) results, in the torsion and in
the torsion-free cases, the proofs are not at all easy.

Therefore the nonsplitting mixed case was not addressed. Since all we know
about the torsion part of such a group is that it is pure in the whole
group, the question of which pure subgroups of perspective groups are also
perspective (just partly addressed) becomes central.

\subsection{Perspective $p$-groups}

First, perspective divisible $p$-groups are described below. For the proof,
recall that the socle of $\mathbb{Z}(p^{\infty })$ is its smallest nonzero
subgroup (having order $p$) and that each subsocle of divisible $p$-group $D$
supports some summand of $D$.

\begin{proposition}
\label{pgru}A divisible $p$-group is perspective iff it has finite rank. As
such it is isomorphic to a finite direct sum of $\mathbb{Z}(p^{\infty })$.
\end{proposition}

\begin{proof}
As already mentioned, an infinite rank direct sum of $\mathbb{Z}(p^{\infty
}) $ is not perspective (see nonexample in the preceding section).
Conversely, let $D$ be a divisible $p$-group. The proof goes by induction on
rank of $D$. If $r(D)=1$ then $D$ is indecomposable and so trivially
perspective.

Let $r(D)=n+1$, $D=A\oplus B=C\oplus U$ and $A\cong C$. We go into several
cases.

\textbf{1)} Let $A+C\neq D$. Then $D=(A+C)\oplus D^{\prime }$ for some $%
D^{\prime }\neq 0$ and so $A$ and $C$ are summands in a divisible group $A+C$
of rank $\leq n$. By induction $A+C=A\oplus K=C\oplus K$ for some $K$,
whence $D=A\oplus (K\oplus D^{\prime })=C\oplus (K\oplus D^{\prime })$.

\textbf{2)} Let $A+C=D$.

a) If $A\cap C=0$ then $D=A\oplus C$, so $B\cong C$. Denote an isomorphism
by $f:B\rightarrow C$. Using Theorem \ref{dia}, the subgroup $B^{\prime
}=\{b+f(b)\,|\,b\in B\}$ is a diagonal, so a summand of $D$ such that $A\cap
D^{\prime }=C\cap B^{\prime }=0$ and $D=A\oplus B^{\prime }=C\oplus
B^{\prime }$.

b) Let $A\cap C\neq 0$, so $(A\cap C)[p]\leq A[p],C[p]$ where $A[p]\cong
C[p] $.

If $(A\cap C)[p]=A[p]$ then also $(A\cap C)[p]=C[p]$. Hence $A[p]=C[p]$, and
by~\cite{fuc1}, $D=A\oplus B=C\oplus B$.

Next assume that $0\neq (A\cap C)[p]<A[p]$, so $(A\cap C)[p]<C[p]$. There
exist a summand $A_{1}$ of $A$ with $A_{1}[p]=(A\cap C)[p]$, $A=A_{1}\oplus
A_{2}$. Similarly $C=C_{1}\oplus C_{2}$, where $C_{1}[p]=(A\cap C)[p]$.
Since $A\cong C$ and $A_{1}[p]=C_{1}[p]$ it follows that $A_{1}\cong C_{1}$
and so $A_{2}\cong C_{2}$.

\textbf{3)} Let $A_{2}+C_{2}\neq D$. Then as in case~\textbf{1)}, $%
D=A_{2}\oplus V=C_{2}\oplus V$ for some $V$ and so $A=A_{2}\oplus (A\cap V)$
and $C=C_{2}\oplus (C\cap V)$. Here $V$ has rank $\leq n$ and $A\cap V$, $%
C\cap V$ are isomorphic summands, so by induction $V=(A\cap V)\oplus
L=(C\cap V)\oplus L$ for some $L$. Finally $D=[A_{2}\oplus (A\cap V)]\oplus
L=[C_{2}\oplus (C\cap V)]\oplus L$.

\textbf{4)} Let $A_{2}+C_{2}=D$. Since $A_{2}\cap C_{2}=0$, as in case 
\textbf{2 a)}, $D=A_{2}\oplus M=C_{2}\oplus M$ for some $M$. Since $r(M)\leq
n$, by induction, $A\cap M$ and $C\cap M$ are perspective in $M$, and as in
case \textbf{3)}, $A$ and $C$ are perspective in $D$.
\end{proof}

If $m$ is a cardinal and $G$ is a group then $G^{(m)}$ denotes the direct
sum of $m$ copies of $G$.

The previous proposition has the following

\begin{corollary}
\label{0.1}For any cardinal $m$ and any natural number $n$, the group $G=%
\mathbb{Z}(p^{n})^{(m)}$ is perspective iff $m$ is finite.
\end{corollary}

\begin{proof}
As already mentioned, an infinite rank direct sum of $\mathbb{Z}(p^{n})$ is
not perspective (see nonexample in the preceding section). Conversely, if $D$
is a divisible hull of $G$ then $G=D[p^{n}]$. If $G=A\oplus B=C\oplus K$, $%
A\cong C$, then by Proposition~\ref{pgru}, $D=D_{A}\oplus U=D_{C}\oplus U$
for some $U\leq D$, where $D_{A}$, $D_{C}$ are the divisible hulls of $A$, $%
C $, respectively. So $G=D[p^{n}]=D_{A}[p^{n}]\oplus
U[p^{n}]=D_{C}[p^{n}]\oplus U[p^{n}]$, where $D_{A}[p^{n}]=A$, $%
D_{C}[p^{n}]=C$.
\end{proof}

Next, about finite or bounded $p$-groups, we have

\begin{proposition}
\label{fini}The finite $p$-groups are perspective.
\end{proposition}

\begin{proof}
The proof goes by induction on the order $|G|$ of the group $G$. Let $%
G=G_{1}\oplus G_{2}$, where $G_{1}$ is a direct sum of finitely many groups $%
\mathbb{Z}(p^{n})$, where $p^{n}$ is the maximal order of elements in $G$, $%
G=A\oplus B=C\oplus K$ and $A\cong C$.

By Corollary \ref{0.1} $G_{1}$ is perspective. Note that if $A\cap G_{1}\neq
0$ then also $C\cap G_{1}\neq 0$. Otherwise, if $C\cap G_{1}=0$, since $%
G_{1} $ is an absolute direct summand~(see \cite{fuc1}, Exercise 8 of \S 9)
of $G$, we can suppose that $C\leq G_{2}$. However, in this case $C$ would
not have any element of order $p^{n}$ but in $A$ such elements exist in view
of $A\cap G_{1}\neq 0$. This would contradict the isomorphism $A\cong C$.

Since in a direct sum of cyclic groups each subsocle supports some summand
of this group~(see \cite{fuc2}, Exercise 3 of \S 66) it follows that $%
A=A_{1}\oplus A_{2}$, $C=C_{1}\oplus C_{2}$, where $A_{1}[p]=(A\cap
G_{1})[p] $, $C_{1}[p]=(C\cap G_{1})[p]$. These direct decompositions of
cyclic groups are isomorphic, so from $A\cong C$ it follows that $A_{1}\cong
C_{1}$ and so $A_{2}\cong C_{2}$. Hence by Corollary \ref{0.1}, $%
G_{1}=A_{1}\oplus U=C_{1}\oplus U$ for some $U$. Then $A=A_{1}\oplus A_{3}$, 
$C=C_{1}\oplus C_{3}$, where $A_{3}=[A\cap (U\oplus G_{2})]$, $C_{3}=[C\cap
(U\oplus G_{2})] $. So $A_{3}$, $C_{3}$ are isomorphic direct summands in $%
U\oplus G_{2}$. Since $|U\oplus G_{2}|<|G|$, by induction $U\oplus
G_{2}=A_{3}\oplus V=C_{3}\oplus V$. Hence $G=G_{1}\oplus G_{2}=(A_{1}\oplus
U)\oplus G_{2}=(A_{1}\oplus A_{3})\oplus V=(C_{1}\oplus C_{3})\oplus V$,
where $A_{1}\oplus A_{3}=A$ and $C_{1}\oplus C_{3}=C$, as desired.
\end{proof}

Since summands of perspective groups are perspective it follows from the
nonexample mentioned before that the Ulm-Kaplanski invariants $f_{n}(G)$ of
perspective reduced $p$-groups $G$ are finite for all integer $n\geq 0$.
Thus a basic subgroup of $G$ is countable and so $|G|\leq 2^{\aleph _{0}}$~%
\cite{fuc1}. Therefore

\begin{corollary}
Any perspective bounded $p$-group is finite.
\end{corollary}

For (Abelian) groups we can introduce the following

\textbf{Definition}. A group is called \textsl{finitely perspective} if it
is perspective with respect to finite (direct) summands. Then we can prove a
surprising (specific for Abelian groups) result.

\begin{proposition}
\label{FP}Each $p$-group $G$ is finitely perspective.
\end{proposition}

\begin{proof}
Let $A\cong C$ be finite summands of $G$. Then $p^{m}A=0$ for some integer $%
m\geq 1$, and so also $p^{m}C=0$ and $p^{m}(A+C)=0$. Hence $A+C\leq H$ for a 
$p^{m}G$-high subgroup $H$. By a theorem of Khabbaz (see \cite{fuc1},
Theorem 27.7), $H$ is a summand of $G=H\oplus F$. We have $H=H_{1}\oplus
\dots \oplus H_{m}$, where $H_{i}$ is a direct sum of groups $\mathbb{Z}%
(p^{i})$ whenever $H_{i}\neq 0$. Let $\pi _{i}G\rightarrow H_{i}$ be the
projections for $i=1,\dots ,m$. Since $A+C$ is finite, it follows that each $%
\pi _{i}(A+C)$ is finite, and $A+C\leq \pi _{1}(A+C)\oplus \dots \oplus \pi
_{m}(A+C)$. Each $\pi _{i}(A+C)$ is contained in some finite summand $%
H_{i}^{\prime }$ of $H_{i}$ and so $H^{\prime }=H_{1}^{\prime }\oplus \dots
\oplus H_{m}^{\prime }$ is a finite summand in $H=H^{\prime }\oplus
H^{\prime \prime }$. By Proposition \ref{fini}, $H^{\prime }=A\oplus
U=C\oplus U$ for some $U$. Then $G=A\oplus (U\oplus H^{\prime \prime }\oplus
F)=C\oplus (U\oplus H^{\prime \prime }\oplus F)$.
\end{proof}

We just mention that if $G=A\oplus C$, where $A\cong C$ then $G=A\oplus
U=C\oplus U$, for any diagonal $U$ with respect to $A$ and $C$.

\bigskip 

Let $\mathfrak{A}$ be a class of (Abelian) groups and $G\in \mathfrak{A}$. A
relativization of our main property can be defined.

\textbf{Definition}. We call $A$ \textsl{perspective in class} $\mathfrak{A}$
if for $G=A\oplus B=C\oplus K$, where $A\cong C$ and $G\in \mathfrak{A}$ it
follows that $G=A\oplus U=C\oplus U$ for some $U$.

Then we can prove a result on torsion-complete (for several equivalent
definitions for reduced groups, see \cite{fuc2}, Theorem 68.4) $p$-groups

\begin{proposition}
The torsion-complete $p$-groups $A$ with finite Ulm-Kaplanski invariants are
perspective in the class of separable $p$-groups.
\end{proposition}

\begin{proof}
Let $G=A\oplus B=C\oplus K$, where $A\cong C$ and let $G$ be a separable $p$%
-group. Then $G=(A_{1}\oplus \dots \oplus A_{n})\oplus (A_{n}^{\ast }\oplus
B)=(C_{1}\oplus \dots \oplus C_{n})\oplus (C_{n}^{\ast }\oplus K)$, where $%
A=(A_{1}\oplus \dots \oplus A_{n})\oplus A_{n}^{\ast }$, $C=(C_{1}\oplus
\dots \oplus C_{n})\oplus C_{n}^{\ast }$, and $A_{1}\oplus \dots \oplus A_{n}
$, $C_{1}\oplus \dots \oplus C_{n}$ respectively, are summands of the basic
subgroups of $A$, $C$ ($A_{k}$, $C_{k}$ are direct sums of cyclic groups of
order $p^{k}$). By Proposition \ref{FP}, $G=(A_{1}\oplus \dots \oplus
A_{n})\oplus U^{(n)}$, $G=(C_{1}\oplus \dots \oplus C_{n})\oplus U^{(n)}$,
where we can choose the $U^{(n)}$'s, such that $U^{(n+1)}$ is a summand in $%
U^{(n)}$ and $U^{(n)}/U^{(n+1)}$ is a direct sum of cyclic groups of order $%
p^{n+1}$. So $G$ has a basic subgroup of type $(\bigoplus_{n\geq
1}A_{n})\oplus (\bigoplus_{n\geq 1}V_{n}^{(n)})=(\bigoplus_{n\geq
1}C_{n})\oplus (\bigoplus_{n\geq 1}V_{n}^{(n)})$, where $V_{n}^{(n)}$ is a
summand in $U_{n}^{(n)}$, each $U_{n}^{(n)}$ is a summand in $U^{(n)}$ and
is a direct sum of cyclic groups of order $p^{n}$. So by \cite{fuc2},Theorem
71.3, $\overline{G}=A\oplus U=C\oplus U$, where $\overline{X}$ is the
torsion completion of $X$ and $A=\overline{(\bigoplus_{n\geq 1}A_{n})}$, $C=%
\overline{(\bigoplus_{n\geq 1}C_{n})}$, $U=\overline{(\bigoplus_{n\geq
1}V_{n}^{(n)})}$. Hence $G=A\oplus (G\cap U)=C\oplus (G\cap U)$, and the
proof is complete.
\end{proof}

\subsection{Perspective torsion-free groups}

As it is well known, the divisible torsion-free groups are the direct sums
of $\mathbb{Q}$.

\begin{proposition}
\label{divtf}A torsion-free divisible group is perspective iff it has finite
rank. As such, it is isomorphic to a finite direct sum of $\mathbb{Q}$.
\end{proposition}

\begin{proof}
As already mentioned, an infinite rank direct sum of $\mathbb{Q}$ is not
perspective. Conversely, for a torsion-free divisible group $D$, the proof
goes by induction on the rank of $D$. Let $r(D)=n+1$, $A$ and $B$ are
isomorphic summands of $D$.

If $A+B<D$ then $D=(A+B)\oplus C$ and $r(A+B)\leq n$, so $A+B=A\oplus
H=B\oplus H$, since by induction the divisible group $A+B$ is perspective.
Then $D=A\oplus(H\oplus C)=B\oplus(H\oplus C)$.

Assume now that $A+B=D$. Only two cases are possible.

1) If $K:=A\cap B\neq 0$ then $K$ is divisible and $D=K\oplus L$ for some $%
L\leq D$. We have $A=K\oplus (A\cap L)$, $B=K\oplus (B\cap L)$. Here $A\cap
L\cong B\cap L$, so by induction $L=(A\cap L)\oplus C$ and $L=(B\cap
L)\oplus C$ whence $D=[K\oplus (A\cap L)]\oplus C$ and $D=[K\oplus (B\cap
L)]\oplus C$, where $[K\oplus (A\cap L)]=A$, $[K\oplus (B\cap L)]=B$.

2) If $A\cap B=0$ then $D=A\oplus B$ and since $A\cong B$ it follows that $A=%
\mathbb{Q}a_{1}\oplus \dots \oplus \mathbb{Q}a_{m}$ and $B=\mathbb{Q}%
b_{1}\oplus \dots \oplus \mathbb{Q}b_{m}$ for some $0\neq a_{1},\dots
,a_{m},b_{1},\dots ,b_{m}\in D$ (with $A$ and $B$ considered as vector
spaces on field $\mathbb{Q}$). Then $D=A\oplus H=B\oplus H$, where $H=%
\mathbb{Q}(a_{1}+b_{1})\oplus \dots \oplus \mathbb{Q}(a_{m}+b_{m})$. \ 
\end{proof}

It was already mentioned that $\mathbb{Z}$ is perspective since it is
indecomposable and that $\mathbb{Z}\oplus \mathbb{Z}$ is not perspective
(see Remark \ref{5.13}). It follows

\begin{proposition}
The only perspective free group is $\mathbb{Z}$.
\end{proposition}

Since the rational groups (the subgroups of $\mathbb{Q}$) are also
indecomposable it follows

\begin{proposition}
The only perspective rank 1 torsion-free groups are the rational groups.
\end{proposition}

Clearly

\begin{proposition}
If a torsion-free group has a summand isomorphic to $\mathbb{Z}\oplus 
\mathbb{Z}$, it is not perspective.
\end{proposition}

Next we characterize some perspective homogeneous completely decomposable
groups.

Note that, similarly to Proposition \ref{fu-inv}, we can prove the following

\begin{proposition}
\label{fu-prod}Let $G=\prod_{i\in I}H_{i}$, where each subgroup $H_{i}$ is
fully invariant in $G$. Then $G$ is perspective iff all $H_{i}$, $i\in I$,
are perspective.
\end{proposition}

Let $p$ be a prime number and $\mathbb{Q}_{p}$ the group (ring) of all
rational numbers with denominator coprime with $p$.

\begin{proposition}
\label{hcd}Let $G$ be a homogeneous completely decomposable group with type $%
\mathbf{t}(G)=\mathbf{t}(\mathbb{Q}_{p})$. The group $G$ is perspective iff $%
G$ has finite rank.
\end{proposition}

\begin{proof}
As already mentioned, only one way needs a proof. Suppose the rank of $G$ is
finite and $G=A\oplus B=C\oplus K$, where $A\cong C$. The proof goes by
induction on $n=r(G)$ the rank of $G$. We distinguish several cases.

1) $A_{1}=A\cap C\neq 0$. Then $A_{1}$, as pure subgroup, is a summand in $G$%
~(see \cite{fuc2}, Lemma \textbf{86.8}), say $A=A_{1}\oplus A_{2}$, $%
C=A_{1}\oplus C_{2}$, where $C_{2}=(A_{2}\oplus B)\cap C$. So $A_{2}\oplus
B=C_{2}\oplus V$ for some $V$. By the induction hypothesis $A_{2}\oplus
U=C_{2}\oplus U$ for some $U$. Hence $G=(A_{1}\oplus A_{2})\oplus
U=(A_{1}\oplus C_{2})\oplus U$, where $A_{1}\oplus A_{2}=A$ and $A_{1}\oplus
C_{2}=C$.

2) $F=B\cap K\neq 0$. Then $B=B^{\prime }\oplus F$ and $K=K^{\prime }\oplus
F $ for some $B^{\prime }$, $K^{\prime }$ and $(A\oplus F)\oplus B^{\prime
}=(C\oplus F)\oplus K^{\prime }$, where $B^{\prime }\cap K^{\prime }=0$. If
now $(A\oplus F)\oplus U=(C\oplus F)\oplus U$ then $A\oplus (F\oplus
U)=C\oplus (F\oplus U)$. So this case reduces to $B\cap K=0$.

3) $A_{K}=A\cap K\neq 0$. Then $A=A_{K}\oplus A_{2}$, $K=A_{K}\oplus K_{2}$
for some $A_{2}$, $K_{2}$ such that $A_{2}\leq C\oplus K_{2}$. Let $\pi
:C\oplus K_{2}\rightarrow C$ be the projection and let $C_{2}=\langle \pi
(A_{2})\rangle $ be the pure hull of $\pi (A_{2})$ in $C$. Since $A_{2}\cap
K_{2}=0$ then $C_{2}\cong A_{2}$ and it follows that $A_{2}\leq C_{2}\oplus
K_{2}$, and so $C_{2}\oplus K_{2}=A_{2}\oplus V$ for some $V$. Since $%
r(C_{2}\oplus K_{2})<n$ by induction hypothesis $C_{2}\oplus U=A_{2}\oplus U$
for some $U$. So $(C_{1}\oplus C_{2})\oplus (A_{K}\oplus U)=(A_{K}\oplus
A_{2})\oplus (C_{1}\oplus U)$, where $A_{K}\cong C_{1}$. Hence if $H$ is a
diagonal in $A_{K}\oplus C_{1}$ (with respect to $A_{K}$ and $C_{1}$) then $%
C\oplus (H\oplus U)=A\oplus (H\oplus U)$. The case $C\cap B\neq 0$ is
similar.

4) $A\cap C=A\cap K=B\cap K=C\cap B=0$. It follows that $r(A)=r(B)=r(K)$.
Let $R=\mathbb{Q}_{p}$. Then $G$ is a free $R$-module of rank $n=2m$. We
have $A=\bigoplus_{i=1}^{m}Ra_{i}$, $B=\bigoplus_{i=1}^{m}Rb_{i}$, $%
C=\bigoplus_{i=1}^{m}Rc_{i}$, and $K=\bigoplus_{i=1}^{m}Rk_{i}$. Since the
pure submodules are summands in $G$, we can choose $c_{i}=a_{i}^{\prime
}+b_{i}^{\prime }$ for some $a_{i}^{\prime }=r_{i}a_{i}$, $b_{i}^{\prime
}=s_{i}b_{i}$, where $r_{i},s_{i}\in R$.

Assume $s_{i}\in pR$, $i=1,\dots ,l$ for some $l\leq m$ and $s_{i}\in
R\setminus pR$ for $i=l+1,\dots ,m$. Note that if $s_{i}\in pR$ then $%
r_{i}\in R\setminus pR$, and similarly if $r_{i}\in pR$ then $s_{i}\in
R\setminus pR$. Let $H=\bigoplus_{i=1}^{l}R(a_{i}+b_{i})$. Then $%
r_{i}^{-1}c_{i}=a_{i}+s_{i}^{\prime }b_{i}$, where $s_{i}^{\prime
}=r_{i}^{-1}s_{i}$ for $i=1,\dots ,l$.

Observe that $H\cap C=0$. Indeed, if $h\in H$ and $h=t_{1}(a_{1}+b_{1})+%
\dots +t_{l}(a_{l}+b_{l})\in C$ for some $t_{1},\dots ,t_{l}\in R$, then $%
[t_{1}(a_{1}+b_{1})+\dots +t_{l}(a_{l}+b_{l})]-[t_{1}(a_{1}+s_{i}^{\prime
}b_{1})+\dots +t_{l}(a_{l}+s_{i}^{\prime }b_{l})]=t_{1}(1-s_{1}^{\prime
})b_{1}+\dots +t_{l}(1-s_{l})b_{l}\in C\cap B=0$. So $t_{1}=\dots =t_{l}=0$,
i.e. $h=0$.

Assume also that $r_{i}\in pR$ for $i=l+1,\dots ,l+r$, when $l+r\leq m$ and $%
r_{i}\in R\setminus pR$ for $i=l+r+1,\dots ,m$ when $l+r<m$. Then as before $%
L\cap (C\oplus H)=0$, where $L=\bigoplus_{i=l+1}^{l+r}R(a_{i}+b_{i})$. If $%
l+r<m$ then let $M=\bigoplus_{i=l+r+1}^{m}R(a_{i}+c_{i})$. We show that $%
A\oplus (H\oplus L\oplus M)=C\oplus (H\oplus L\oplus M)=G$. Indeed, this
follows from the fact that these sums of summands are direct, as for the
construction we used linear independent subsystems and these sums equal $G$
since these contain the basis of $G$.

From the construction of $H$ and $L$ it follows that $b_{i}\in A\oplus
H\oplus L$ for $i=1,\dots ,l+r$. Since $c_{i}=r_{i}a_{i}+s_{i}b_{i}$, where $%
s_{i}$ are invertible in $R$ for $i=l+r+1,\dots ,m$ it follows that $%
b_{i}\in A\oplus M$ for the specified $i$. Hence $G=A\oplus (H\oplus L\oplus
M)$. It remains to prove that $a_{i},b_{i}\in C\oplus (H\oplus L\oplus M)$
for all $i=1,\dots ,m$. Indeed, if for example, $i=1,\dots ,l$, then $%
c_{i}-(r_{i}a_{i}+r_{i}b_{i})=(r_{i}a_{i}+s_{i}b_{i})-(r_{i}a_{i}+r_{i}b_{i})=(s_{i}-r_{i})b_{i}\in C\oplus H 
$. Since $s_{i}\in pR$ and $r_{i}\in R\setminus pR$ we get $s_{i}-r_{i}\in
R\setminus pR$, so this element is invertible in $R$, whence $b_{i}$ and so $%
a_{i}$ belong to $C\oplus H\leq C\oplus (H\oplus L\oplus M)$ for $i=1,\dots
,l$. Similarly we can prove that all $a_{i},b_{i}\in C\oplus (H\oplus
L\oplus M)$ for $i=1,\dots ,m$, i.e. $G=C\oplus (H\oplus L\oplus M)$. Since
the sum of the ranks of $H$, $L$, $M$ equals $m$ the latter sum is direct.
\end{proof}

\textbf{Remark}. Since the pure subgroups of group $G$ in the previous
proposition are summands of $G$, this is a trivial example of groups with
all pure subgroups being perspective.

\bigskip

The next result draws attention on rank 2 summands of homogeneous
torsion-free groups of finite rank.

\begin{lemma}
If $G$ is a torsion-free homogeneous group of finite rank then $G$ is
perspective iff $G$ has a perspective rank $2$ summand.
\end{lemma}

\begin{proof}
The condition is necessary since summands of perspective groups are
perspective.

Conversely, since equal rank summands of $G$ are isomorphic, it suffices to
focus on any summand. Let $G=A\oplus B=C\oplus K$, where $A\cong C$. As in
Proposition \ref{hcd}, we can assume that $A\cap C=A\cap K=0$. Let $0\neq
a\in A$ and $a=c+x$, where $c\in C$, $x\in K$ ($c,x\neq 0$). If $%
A_{1}=\langle a\rangle _{\ast }$, $C_{1}=\langle c\rangle _{\ast }$, $%
K_{1}=\langle x\rangle _{\ast }$ then $A_{1}$, $C_{1}$ and $K_{1}$ are (rank
1) summands, so $A=A_{1}\oplus A_{2}$, $C=A_{1}\oplus C_{2}$, $K=K_{1}\oplus
K_{2}$ and $G=(C_{1}\oplus K_{1})\oplus (C_{2}\oplus K_{2})$. By hypothesis, 
$C_{1}\oplus K_{1}=C_{1}\oplus W_{1}=A_{1}\oplus W_{1}$ for some $W_{1}$, so 
$G=C_{1}\oplus (W_{1}\oplus C_{2}\oplus K_{2})=A_{1}\oplus (W_{1}\oplus
C_{2}\oplus K_{2})$. Since $A=A_{1}\oplus A\cap (W_{1}\oplus C_{2}\oplus
K_{2})$, consider$A_{2}=A\cap (W_{1}\oplus C_{2}\oplus K_{2})$. Hence $%
W_{1}\oplus C_{2}\oplus K_{2}=A_{2}\oplus L$ for some $L$ and we complete
the proof by induction: $W_{1}\oplus C_{2}\oplus K_{2}=C_{2}\oplus
W=A_{2}\oplus W$ so $G=C_{1}\oplus (W_{1}\oplus C_{2}\oplus
K_{2})=(C_{1}\oplus C_{2})\oplus W=A_{1}\oplus (W_{1}\oplus C_{2}\oplus
K_{2})=(A_{1}\oplus A_{2})\oplus W$.
\end{proof}

Using Proposition \ref{fu-inv} and Proposition \ref{fu-prod} it follows that

\begin{corollary}
Let $G=\bigoplus_{p\in P}G_{p}$ ($G=\prod_{p\in P}G_{p}$), where $P$ is some
subset of prime numbers and $G_{p}$ are homogeneous completely decomposable
groups of finite rank with type $t(G_{p})=t(\mathbb{Q}_{p})$. Then the group 
$G$ is perspective.
\end{corollary}

Next, we describe the perspective torsion-free algebraically compact groups.
Denote by $\mathbb{\hat{Z}}_{p}$ be the ring (group) of $p$-adic integers
and by $\mathbb{P}$ the set of all prime numbers.

\begin{proposition}
\label{tfac}A non-zero torsion-free algebraically compact group $G$ is
perspective iff $G=\prod_{p\in \pi }G_{p}$, where $\varnothing \neq \pi
\subseteq \mathbb{P}$ and $G_{p}$ is a finite direct product of copies of
the group $\mathbb{\hat{Z}}_{p}$.
\end{proposition}

\begin{proof}
To show that the condition is necessary, first recall that any torsion-free
algebraically compact group $G$ has the form $G=\prod_{p\in \pi }G_{p}$,
where $G_{p}$ is a $p$-adic algebraically compact group~(see \cite{fuc1},
Proposition \textbf{40.1}), and in particular, is $\mathbb{\hat{Z}}_{p}$%
-module. So the rank of each $G_{p}$ is finite, i.e. $G_{p}$ is a free $%
\mathbb{\hat{Z}}_{p}$-module of finite rank.

To show that the condition is sufficient, also recall that intersections of
summands in torsion-free $p$-adic algebraically compact groups are also
summands of this group. Next note that in the ring $\mathbb{\hat{Z}}_{p}$,
all elements of $\mathbb{\hat{Z}}_{p}\setminus p\mathbb{\hat{Z}}_{p}$ are
invertible, and that in any $\mathbb{\hat{Z}}_{p}$-module of finite rank,
pure submodules are summands. As in Proposition \ref{hcd}, one can prove
that $p$-adic algebraically compact modules of finite rank are perspective.
Then using Proposition \ref{fu-prod}, the proof is complete.
\end{proof}

\begin{example}
Let $G=G_{1}\oplus \dots \oplus G_{n}$, where $G_{i}$, $1\leq i\leq n$ are
perspective groups and $\mathrm{Hom}(G_{i},G_{j})=0$ for $i=2,\dots ,n$ and $%
1\leq j\leq i-1$. Then $G$ is perspective.
\end{example}

\begin{proof}
The proof goes by induction on $n$. If $G_{2}\oplus \dots \oplus G_{n}$ is
perspective then since it is fully invariant in $G$, it is perspective by
Corollary \ref{5.10}.
\end{proof}

Next, we give some examples and nonexamples.

Recall that a torsion-free group $G$ is called \textsl{cohesive} if $G/H$ is
divisible for all pure subgroups $H\neq 0$ of $G$. For some facts about such
groups see \cite{fuc2}, \textbf{88}, exercise 17.

\begin{example}
A perspective direct sum of pure subgroups of a perspective group.

\normalfont By Proposition \ref{tfac}, the group $G=\prod_{p\in \mathbb{P}}%
\mathbb{\hat{Z}}_{p}$ is perspective. Let $A_{i}$, $i=1,\dots ,n$, be pure
subgroups of $G$ such that $p^{\omega }A_{i}=0$ for all $p\in \mathbb{P}$
and $r(A_{i})=m_{i}$, where $1\leq m_{1}<\dots <m_{n}\leq 2^{\aleph _{0}}$.
Then $A_{i}$ are cohesive groups~(see \cite{KMT}, \textbf{32}). If $%
A=A_{1}\oplus \dots \oplus A_{n}$, then $A$ is perspective.

Indeed, each $A_{i}$ is perspective as indecomposable group and $%
\bigoplus_{j=i}^{n}A_{j}$ is fully invariant in $A$ for all $i=2,\dots ,n-1$%
. This follows from the fact that $\mathrm{Hom}(A_{j},A_{i})=0$ for every $%
i<j$, owing that since $r(A_{i})<r(A_{j})$, each homomorphism $%
f:A_{j}\rightarrow A_{i}$ has a non-zero kernel and so $f(A_{j})$ is
divisible, whence $f=0$.
\end{example}

In general we can ask

\textbf{Question.} Which \emph{pure subgroups} of a perspective group are
perspective ?

\begin{example}
A subgroup of a perspective group which is not perspective.

\normalfont Let $G$ be the torsion-free group of rank $n\geq 3$ as in~\cite%
{fuc2}, \textbf{88}, exercise 8. Then $G$ is perspective as an
indecomposable group but all the subgroups in $G$ of rank $n-1$ are free. So
the proper subgroups of rank $\geq 2$ of $G$ are not perspective.
\end{example}

\begin{example}
A factor group of perspective groups which is not perspective.

\normalfont Let $G$ be a torsion-free cohesive group with $r(G)\geq \aleph
_{0}$ and let $H$ be a pure subgroup of $G$ such that $r(G/H)\geq \aleph
_{0} $ (e.g., a pure subgroup $H$ of rank 1). Then $G/H$ is not perspective
being divisible torsion-free group of infinite rank, but $G$ and $H$ are
perspective being indecomposable groups.
\end{example}

\begin{example}
A factor group of a not perspective group may be perspective.

\normalfont Let $X=\{a_{n}:n\in \mathbb{N}^{\ast }\}$ and let $%
G=\left\langle X\right\rangle $ be free of countable rank. Consider the
function $f:X\rightarrow \mathbb{Q}$, $f(a_{n})=\frac{1}{n!}$ for every $%
n\in \mathbb{N}^{\ast }$. The group $G$ being free, $f$ extends to a group
homomorphism $\overline{f}:G\rightarrow \mathbb{Q}$, obviously surjective
(as $\mathbb{Q}=\left\langle \frac{1}{n!}:n\in \mathbb{N}^{\ast
}\right\rangle $). So $G$ is not perspective but $\mathbb{Q}\cong G/\ker (%
\overline{f})$ is perspective.
\end{example}

As in the module case, conditions for a perspective group $G$ which assure
that $G\oplus G$ is also perspective, are of interest (and difficult to
find).

In closing, we address this problem for rank 1 torsion-free groups, which
are perspective as indecomposable groups.

First, a simple lemma which is well known

\begin{lemma}
\label{nG}Each subgroup $A$ of torsion-free group $G$ of rank $1$,
isomorphic to the whole group, is of form $nG$, for some integer $n\geq 1$.
\end{lemma}

\begin{proof}
If $f:G\rightarrow A$ is an isomorphism then $f$ acts as multiplication with
some rational number $n/m$, where $mG=G$, and so $f(G)=nG=A$.
\end{proof}

Next, a characterization.

\begin{proposition}
\label{G+G}Let $G$ be a torsion-free group of rank $1$. Then $G\oplus G$ is
perspective iff for all coprime integers $m,n\geq 1$ and all integers $%
k,t\geq 0$, such that

(i) at least one of $k,t$ is non-zero, and

(ii) if one of $k,t$ is zero then the other is equal to $1$, and

(iii) if both $k,t$ are non-zero these are coprime,

in each of the following cases:

1) $mG\neq G$, $k=1$, $t=0$,

2) $nG\neq G$, $k=0$, $t=1$,

3) $mG\neq G$, $tG\neq G$, where $t\neq 0$,

4) $mG,nG\neq G,$

there exist coprime integers $s$, $l$ with $(ml-sn)G=G$, $(kl-st)G=G$.
\end{proposition}

\begin{proof}
To show that the conditions are necessary, we present $G\oplus G$ as $%
F=Ra\oplus bR$, where $R$ is an additive subgroup with $1$ of $\mathbb{Q}$,
isomorphic to $G$. Since $m$, $n$ are coprime, the subgroup $R(ma+nb)$ is
pure, so its is a summand of $F$. Let $\pi :F\rightarrow Ra$ and $\theta
:F\rightarrow Rb$ be the projections and let $A$ and $C$ be isomorphic rank
1 summands of $F$. Hence $A=R(ma+nb)$ and $C=R(ka+tb)$ for some $m,n,k,t\in 
\mathbb{Z}$, where at least one of $\{m,n\}$ is $1$ and the corresponding
integer of $\{k,t\}$ is non-zero. It suffices to consider the case $m,n\neq
0 $. Since $A$ is a summand of $F$, $dG=G$ for $d=\gcd (m,n)$, so we can
consider $d=1$. Similarly, if $t=0$ and $C=Rka$ then $kG=G$, so we can
assume $k=1$.

1) Let $mG\neq G$. Assume that $F=R(ma+nb)\oplus U=Ra\oplus U$ for some $U$,
i.e. in this case $k=1$, $t=0$. Clearly $U\neq Ra,Rb$ and so $\pi (U),\theta
(U)\neq 0$, whence $\pi (U)=Rsa$, $\theta (U)=Rlb$ for some coprime integers 
$s,l\geq 1$ (by Lemma \ref{nG}) and so $U=R(sa+lb)$. We have $%
-s(ma+nb)+m(sa+lb)=(ml-sn)b$ and $-sa+(sa+lb)=lb$. Since the presentation of
elements is unique in direct sums and $Rb\leq F$, it follows that $%
(ml-sn)G=lG=G$.

2) Let $nG\neq G$. Assume that $F=R(ma+nb)\oplus U=Rb\oplus U$ for some $U$,
i.e. in this case $t=1$, $k=0$. Clearly $U\neq Ra,Rb$ and as above we can
show that $(ml-sn)G=sG=G$ for some $m,l\in \mathbb{Z}$ with coprime $s\,$, $%
l $.

3) Let $mG\neq G$, $tG\neq G$ for $t\neq 0$. Clearly $U\neq Ra,Rb$, so $%
U=R(sa+lb)$, with coprime $s$, $l$. We have $-s(ma+nb)+m(sa+lb)=(ml-sn)b$
and $-s(ka+tb)+k(sa+lb)=(kl-st)b$. Hence $(ml-sn)G=G$, $(kl-st)G=G$.

4) Let $mG,nG\neq G$. Then $U\neq Ra,Rb$, so, as in case 3), $(ml-sn)G=G$, $%
(kl-st)G=G$ for some coprime $s,l\in \mathbb{Z}$.

To show that the conditions are sufficient, let $F=A\oplus B=C\oplus K$, $%
A\cong C$ and $r(A)=1$. We consider several cases.

\textbf{I}. a) $A=Ra$, $C=Rb$. We can choose $U=R(a+b)$.

b) $A=R(m^{\prime }a+n^{\prime }b)$, where $m^{\prime }G=n^{\prime }G=G$. If 
$C=Rb$ we can take $U=Ra$, and if $C=Ra$ we can take $U=Rb$.

c) $A=R(m^{\prime }a+n^{\prime }b)$, $C=R(k^{\prime }a+t^{\prime }b)$, where 
$m^{\prime }G=n^{\prime }G=G$, $k^{\prime }G=t^{\prime }G=G$. Then we can
choose $U=Ra$ or $U=Rb$.

\textbf{II}. $A=R(ma+nb)$, where $mG\neq G$ or $nG\neq G$.

If $A=R(ma+nb)$, $C=R(ka+tb)$, $mG\neq G$, $nG=G$, $tG=G$ and $kG=G$ or $%
kG\neq G$, $k\neq 0$, then in both cases we can take $U=Ra$. Also $U=Ra$ if $%
C=Rb$.

1) $A=R(ma+nb)$, where $mG\neq G$ and $C=Ra$, i.e. $k=1$, $t=0$. Since $%
mG\neq G$ then $n\neq 0$ (otherwise $A$ is not summand of $F$). Then by
hypothesis there exist coprime $s,l\in \mathbb{Z}$ with $(ml-sn)G=G$, $lG=G$%
. If $U=R(sa+lb)$ then $-s(ma+nb)+m(sa+lb)=(ml-sn)b$, $-sa+(sa+lb)=lb$ and $%
-l(ma+nb)+n(sa+lb)=(ns-lm)a$, where $(ml-sn)G=G$ and $lG=G$, so $Ra,Rb\leq
A\oplus U,C\oplus U$ whence $F=A\oplus U=C\oplus U$.

2) $A=R(ma+nb)$, where $nG\neq G$, so $m\neq 0$, and $C=Rb$, i.e. $k=0$, $%
t=1 $. Let $s,l\in \mathbb{Z}$ be coprime, $(ml-sn)G=G$, $tG=G$ and $%
U=R(sa+lb)$. As in the previous case 1), $F=A\oplus U=C\oplus U$. The
remaining cases 3) and 4) are similar and since this way all the possible
cases are covered, the proof is complete.
\end{proof}

\begin{corollary}
Let $G$ be a torsion-free group of rank $1$ such that $G\oplus G$ is
perspective. Then $G$ is $p$-divisible at least for one prime number $p$.
\end{corollary}

\begin{proof}
As in Proposition \ref{G+G}, assume $(ml-sn)G=G$ and $(kl-st)G=G$. Moreover,
assume that $ml-sn=\pm 1$ and $kl-st=\pm 1$. Then if $t=0$ we have $l=\pm 1$
and so $\pm m\pm 1=sn$. Since we can choose coprime $m$ and $n$ such that $%
(\pm m\pm 1)/n\notin \mathbb{Z}$, it follows $G$ is not divisible only by $%
\pm 1$.
\end{proof}

The converse fails as shows the following

\begin{example}
If $G$ is a torsion-free group of rank~$1$, with $2G,5G\neq G$ and $G$ is
divisible only by~$11$, then $G\oplus G$ is not perspective.

\normalfont By contradiction, suppose $G\oplus G$ is perspective. Then
according to Proposition \ref{G+G}, $5l-2s=\pm 11^{a}$, $kl-st=\pm 11^{b}$,
for some integers $a,b\geq 0$. Taking $l=0$ we can suppose $t=1$ and so $%
s=\pm 11^{b}$ and $5l\pm 2\cdot 11^{b}=\pm 11^{a}$. Since $(l,11^{b})=1$ we
get $b=0$ or $a=0$. If $b=0$ then the equation $5l\pm 2=\pm 11^{a}$ has no
solutions since the last digit of the RHS 1 but is 2 or 7 in the LHS. If $a=0
$ then the equation $5l\pm 1=\pm 2\cdot 11^{b}$ has no solutions since the
last digit of the RHS is 2 but is 1 or 6 in the LHS.
\end{example}

A result of the same sort is the following

\begin{proposition}
If $G$ is a torsion-free homogeneous of rank $1$ group such that $G$ is
divisible for all prime numbers except two coprime numbers $p$ and $q$ then $%
G\oplus G$ is perspective.
\end{proposition}

\begin{proof}
Let $F=G\oplus G$ and $F=A\oplus B=C\oplus K$, where $A\cong C$, $r(A)=1$.
As in the previous proposition \ref{G+G}, we take $A=R(ma+nb)$, $C=R(ka+tb)$%
, and in view of the sufficiency part,~we can suppose $m,n,k,t\neq 0$. We
are searching $U$ such that $U=R(sa+lb)$, where $(s,l)=1$. Consider all the
possible cases with respect to divisibility of $G$ by $m,n,k,t$ (in the next
table, the sign "+" means divisibility by the corresponding number, sign
"--" \ not divisibility).

{\tiny {\ 
\begin{tabular}[t]{|c|c|c|c|c|c|c|c|c|c|c|c|c|}
\hline
& 1 & 2 & 3 & 4 & 5 & 6 & 7 & 8 & 9 & 10 & 11 & 12 \\ \hline
$m$ & + & + & + & + & + & $-$ & $-$ & $-$ & $-$ & + & $-$ & $-$ \\ 
$n$ & + & + & + & + & $-$ & + & $-$ & $-$ & $-$ & $-$ & + & + \\ 
$k$ & + & + & $-$ & $-$ & + & + & $-$ & $-$ & + & + & + & $-$ \\ 
$t$ & + & $-$ & + & $-$ & + & + & $-$ & + & + & $-$ & $-$ & + \\ \hline
\end{tabular}
} }

As the cases $\{2,3,5,6\}$, $\{4,9\}$ and $\{10,12\}$ are respectively
similar, it suffices to check the cases 2, 4 and 10.

1)--2) $mG=nG=kG=G$ and $tG=G$ or $tG\neq G$. If $U=Rb$ then $%
F=R(ma+ng)\oplus U=R(ka+tb)\oplus U$.

4) $mG=nG=G$, and $kG,tG\neq G$. Since $\gcd (k,t)=1$, let $p\mid k$, $q\mid
t$ and $q\nmid k$, $p\nmid t$. If now $U=R(qa+pb)$, i.e. $s=q$, $l=p$, then $%
p,q\nmid (mp-qn)$ and $p,q\nmid (kp-qt)$, so $(mp-qn)G=G$, $(kp-qt)G=G$.
Consequently, by Proposition \ref{G+G}, $F=A\oplus U=C\oplus U$. If $p\mid t$%
, $q\mid k$, where $q\nmid t$, $p\nmid k$ and $U=R(pa+qb)$ then $p,q\nmid
mq-pn$, $p,q\nmid kq-pt$.

7) $mG,nG,kG,tG\neq G$.

a) $p\mid m,t$ and $q\mid n,k$. If $U=R(a+b)$ then $p,q\nmid m-n$ and $%
p,q\nmid k-t$.

b) $p\mid m,k$ and $q\mid n,t$. Let $U=R(qa+pb)$, then $p,q\nmid mp-qn$ and $%
p,q\nmid kp-qt$.

8) $p\mid m,k$ and $q\mid n$, $tG=G$.

a) $q\nmid k$. If $U=R(qa+b)$ then $p,q\nmid m-qn$ and $p,q\nmid k-qt$.

b) $q\mid k$. If $U=R(a+b)$, then $p,q\nmid m-n$ and $p,q\nmid k-t$.

10) $mG,kG=G$ and $nG,tG\neq G$. If $U=Rb$ then $F=A\oplus U=C\oplus U$.

11) a) $q\mid m,t$ and $p\nmid m,t$, but $nG=kG=G$. Let $U=R(pa+b)$, i.e. $%
s=p$, $l=1$. Then $p,q\nmid m-pn$ and $p,q\nmid k-pt$.

b) $p,q\mid m,t$. If $U=R(a+b)$ then $p,q\nmid m-n$, and $p,q\nmid k-t$.

c) $p\nmid m$, $q\mid m$ and $q\nmid t$, $p\mid t$. If $U=R(pa+qb)$, i.e. $%
s=p$, $l=q$ then $p,q\nmid mq-pn$ and $p,q\nmid kq-pt$.
\end{proof}

\begin{acknowledgement}
\normalfont The second author was supported by the Ministry of Science and Higher
Education of Russia (agreement No. 075-02-2023-943
\end{acknowledgement}


\bigskip

\begin{thebibliography}{9}
\bibitem{che} A. R. Chekhlov \textsl{Abelian groups with normal endomorphism
rings}. Algebra and Logic \textbf{48} (2009), 298-308.

\bibitem{fuc1} L. Fuchs \textsl{Infinite abelian groups}, vol. 1, Pure and
Applied Mathematics, vol. 36, Academic Press, New York, London 1970.

\bibitem{fuc2} L. Fuchs \textsl{Infinite abelian groups}, vol. 2, Pure and
Applied Mathematics, vol. 36-II, Academic Press, New York, London 1973.

\bibitem{gar} S. Garg, H. K. Grover, D. Khurana \textsl{Perspective rings}.
J of Algebra \textbf{415} (2014), 1-12.

\bibitem{goo} K.R. Goodearl \textsl{Von Neumann Regular Rings}. Krieger
Publishing Comp., Second Edition (1991).

\bibitem{khu} D. Khurana, T. Y. Lam \textsl{Rings with internal cancellation.%
} J. Algebra \textbf{284} (2005), 203-235.

\bibitem{KMT} P. A. Krylov, A. V. Mikhalev, A. A. Tuganbaev \textsl{%
Endomorphism rings of abelian groups}. Algebras and Applications, vol. 2.
Kluwer Academic Publishers, Dordrecht, 2003.

\bibitem{lam1} T. Y. Lam \textsl{A crash course on stable range,
cancellation, substitution and exchange}. J. of Algebra and Appl. \textbf{3}
(3) (2004), 301-343.

\bibitem{sch} R. Schmidt \textsl{Subgroup lattices of groups}. de Gruyter
Expositions in Mathematics, 14. Walter de Gruyter, 1994.
\end{thebibliography}
\end{document}